\documentclass{amsart}

\usepackage[english]{babel}
\usepackage[utf8x]{inputenc}
\usepackage[T1]{fontenc}
\usepackage{comment}

\usepackage{amsmath}
\usepackage{amssymb}
\usepackage{amsthm}
\usepackage{graphicx}
\usepackage[colorinlistoftodos]{todonotes}
\usepackage[colorlinks=true, allcolors=blue]{hyperref}
\usepackage{url}

\newcommand{\NN}{\mathbb{N}}
\newcommand{\ZZ}{\mathbb{Z}}

\newcommand{\ra}{\rightarrow}
\renewcommand{\aa}{\alpha}
\newcommand{\bb}{\beta}

\newtheorem{thm}{Theorem}

\theoremstyle{definition}
\newtheorem{theorem}[thm]{Theorem}
\newtheorem{lemma}[thm]{Lemma}
\newtheorem{definition}[thm]{Definition}

\newtheorem{proposition}[thm]{Proposition}
\newtheorem{remark}[thm]{Remark}

\newtheorem{corollary}[thm]{Corollary}

\numberwithin{thm}{section}

\title[$(s,s+2)$-core partitions with distinct parts]
{A bijective proof of Amdeberhan's conjecture on the number of $(s, s+2)$-core partitions with distinct parts}

\author{Jineon Baek}
\address{Jineon Baek, University of Michigan, Department of Mathematics,  
2074 East Hall,
530 Church Street,
Ann Arbor, MI 48109-1043}
\email{jineon@umich.edu}

\author{Hayan Nam}
\address{Hayan Nam, University of California, Irvine, Department of
Mathematics, 340 Rowland Hall, Irvine, CA 92697}
\email{hayann@uci.edu}

\author{Myungjun Yu}
\address{Myungjun Yu, University of Michigan, Department of Mathematics,  
2074 East Hall,
530 Church Street,
Ann Arbor, MI 48109-1043}
\email{myungjuy@umich.edu}

\begin{document}
\maketitle
\sloppy

\begin{abstract}
Amdeberhan conjectured that the number of $(s,s+2)$-core partitions with distinct parts for an odd integer $s$ is $2^{s-1}$. This conjecture was first proved by Yan, Qin, Jin and Zhou, then subsequently by Zaleski and Zeilberger. Since the formula for the number of such core partitions is so simple one can hope for a bijective proof. We give the first direct bijective proof of this fact by establishing a bijection between the set of $(s, s+2)$-core partitions with distinct parts and a set of lattice paths.
\end{abstract}

\section{Introduction}

A positive integer tuple $\lambda = (\lambda_1, \lambda_2, \ldots, \lambda_{\ell})$ is a partition of $n$ if $\sum_{i=1}^{\ell}\lambda_i = n$ and $\lambda_i$ is weakly decreasing. We visualize a partition using a Ferrers diagram (Figure \ref{ferres-diagram}). Each square in a Ferrers diagram is called a cell. We define the \emph{hook length} of a cell as the sum of the number of cells on its right, the number of cells below it, and 1 for itself. \\

\begin{figure}[h]
\centering
\begin{tikzpicture}[scale=0.50]
\fill [orange] (3, 2) rectangle (4, 3);
\fill [orange] (3, 3) rectangle (4, 4);
\fill [orange] (4, 3) rectangle (5, 4);
\fill [orange] (5, 3) rectangle (6, 4);
\fill [orange] (6, 3) rectangle (7, 4);
\draw (0,4) -- (6,4);
\draw (0,0) -- (0,4);
\draw (0,0) -- (1,0);
\draw (1,0) -- (1,4);
\draw (0,1) -- (3,1);
\draw (0,2) -- (6,2);
\draw (1,1) -- (1,4);
\draw (2,1) -- (2,4);
\draw (3,1) -- (3,4);
\draw (4,2) -- (4,4);
\draw (6,2) -- (6,4);
\draw (0,3) -- (6,3);
\draw (5,2) -- (5,4);
\draw (6,4) -- (7,4) -- (7,3) -- (6,3);
\node at (6.5, 3.5) {1};
\node at (5.5, 3.5) {3};
\node at (5.5, 2.5) {1};
\node at (4.5, 3.5) {4};
\node at (4.5, 2.5) {2};
\node at (3.5, 3.5) {5};
\node at (3.5, 2.5) {3};
\node at (2.5, 3.5) {7};
\node at (2.5, 2.5) {5};
\node at (2.5, 1.5) {1};
\node at (1.5, 3.5) {8};
\node at (1.5, 2.5) {6};
\node at (1.5, 1.5) {2};
\node at (0.5, 3.5) {10};
\node at (0.5, 2.5) {8};
\node at (0.5, 1.5) {4};
\node at (0.5, 0.5) {1};
\end{tikzpicture}
\caption{A Ferrers diagram of $\lambda = (7, 6, 3, 1)$ with the hook length of each cell.}
\label{ferres-diagram}
\end{figure}

A partition is called an \emph{$a$-core partition} (or simply an \emph{$a$-core}) if none of its hook lengths are divisible by $a$. For example, $\lambda = (7, 6, 3, 1)$ is a 9-core since no hook length is a multiple of 9. It is also an $a$-core for $a > 10$. If a partition is both an $a$-core and a $b$-core, then we say that the partition is an $(a,b)$-core. Similarly, we also say that a partition is an $(a_1, a_2, \ldots, a_n)$-core if it is simultaneously an $a_1$-core, an $a_2$-core, \ldots, and an $a_n$-core.

It is well-known that the number of $a$-core partitions is infinite. For example, for every positive integer $k$, the partition $\lambda = (k(a - 1) + 1, \ldots, 2(a - 1) + 1, (a - 1) + 1, 1)$ is an $a$-core. Anderson \cite{And02} has proved that the number of $(a,b)$-core partitions is equal to $\mathbf{Cat}_{a,b} = \frac{1}{a+b}{a+b \choose a}$ when $(a,b) =1$. After Anderson \cite{And02}, counting simultaneous core partitions has been a fascinating subject and studied by many authors \cite{article1}, \cite{article3}, \cite{Amd16}, \cite{article7}, \cite{article8}, \cite{article9}, \cite{article2}, \cite{article4}, \cite{article5}.


We say a partition $\lambda$ has \emph{distinct parts} if all its parts $\lambda_1, \lambda_2, \cdots, \lambda_{\ell}$ of $\lambda$ are different from each other. Straub \cite{Str1} and Xiong \cite{Xiong} proved that the number of $(s,s+1)$-core partitions with distinct parts is equal to the $(s+1)\textsuperscript{st}$ Fibonacci number. Amdeberhan conjectured that when $s$ is odd, the number of $(s, s+2)$-core partitions with distinct parts is equal to $2^{s-1}$ (\cite[Conjecture 11.10]{Amd16}).
This was subsequently proved
by Yan, Qin, Jin and Zhou \cite{YQJZ} by an ingenious idea combining bijection arguments with manipulation of formulas involving binomials and Catalan numbers. Then Zaleski and Zeilberger \cite{ZZ}  gave another proof of this fact and found the first 22 moments of the distribution of sizes of the $(s, s+2)$-cores with distinct parts by using their symbolic-computational algorithms.

Despite the simplicity of the formula for the number of $(s,s+2)$-cores with distinct parts, no bijective proof has appeared in the literature. Inspired by Yan, Qin, Jin and Zhou's argument in \cite{YQJZ}, we give the first direct bijective proof of this fact, identifying $(s, s+2)$-core partitions with distinct parts with a particular kind of lattice paths.
We define a \emph{lattice path} of length $n$ to be a path in $\ZZ \times \ZZ$ from $(0, 0)$ to $(n, y)$ for some $y \in \ZZ$ using up steps $U = (1, 1)$ and down steps $D = (1, −1)$. Our main result is the following.
\begin{theorem}
\label{main-bijection}
For every odd integer $s > 0$,
there is a bijection between 
the set of $(s, s+2)$-core partitions with distinct parts
and the set of lattice paths of length $s$ from $(0, 0)$ to $(s, y)$, where $y$ varies over the positive integers.
\end{theorem}
Because the number of lattice paths in Theorem \ref{main-bijection} is exactly the half of the number of lattice paths of length $s$, the number of $(s, s+2)$-core partitions with distinct parts then follows as a corollary.
\begin{corollary}
\label{Amdeberhan-conj}
For every odd integer $s > 0$, the number of $(s, s + 2)$-core partitions with distinct parts is $2^{s−1}$.
\end{corollary}

We use the well-known identification of a core partition with its $\beta$-set to rephrase Theorem \ref{main-bijection}. For a partition $\lambda$, the \emph{$\beta$-set} of $\lambda$ (denoted by $\beta(\lambda)$) is the sequence of hook
lengths of the boxes in the first column of $\lambda$.
For example, the $\beta$-set of the partition $\lambda = (7, 6, 3, 1)$ in Figure \ref{ferres-diagram} is $\beta(\lambda) = \{10, 8, 4, 1\}$. Given a (weakly) decreasing sequence of positive integers $h_1, h_2, \cdots, h_m$, it is
easily seen that the unique partition $\lambda$ with $\beta(\lambda) = \{h_1, h_2, \cdots, h_m\}$ is
$$\lambda = (h_1 − (m − 1), h_2 − (m − 2), . . . , h_{m−1} − 1, h_m).$$

Let $\NN$ be the set of non-negative integers.
Given a poset $(P, \le)$, an \emph{order ideal} $I$ of $P$ is a subset of $P$ such that for any $y \in I$, $x \leq y$ implies $x \in I$. Let
$$
P_{s,t} = \NN \setminus \{n \in \NN \mid n = a s + b t \textnormal{ for some } a, b \in \NN\}
$$
where the partial order is given by requiring that $x \in P_{s,t}$ \emph{covers} $y \in P_{s,t}$ if and only if $x−y$ is either $s$ or $t$. The map $\lambda \mapsto \beta(\lambda)$ from $(s, t)$-core partitions $\lambda$ to the order ideals $I = \beta(\lambda)$ of the poset $P_{s,t}$ is a bijection \cite{YQJZ}. The following characterization of $\beta$-sets of partitions with distinct parts was used in \cite{Xiong} and \cite{YQJZ}.
\begin{lemma}
\label{cond-dist-parts}
A partition $\lambda$ is a partition into distinct parts if and only if none of $x,y \in \beta(\lambda)$ satisfy $x-y=1$. 
\end{lemma}
Let $\lambda$ be an $(s, s+2)$-core partition. The $\beta$-set of $\lambda$ gives an order ideal of $P_{s, s+2}$ containing no two adjacent numbers. For example, a $(9, 11)$-core partition $\lambda = (15, 7, 6, 3, 1)$ corresponds to the order ideal $I = \{19, 10, 8, 4, 1\}$ of $P_{9, 11}$. Figure \ref{order-ideal-example} denotes the Hasse diagram of 
$P_{9, 11}$ with the elements of the order ideal $I = \{19, 10, 8, 4, 1\}$ as white circles.

\begin{figure}[h]
\includegraphics[width=6cm]{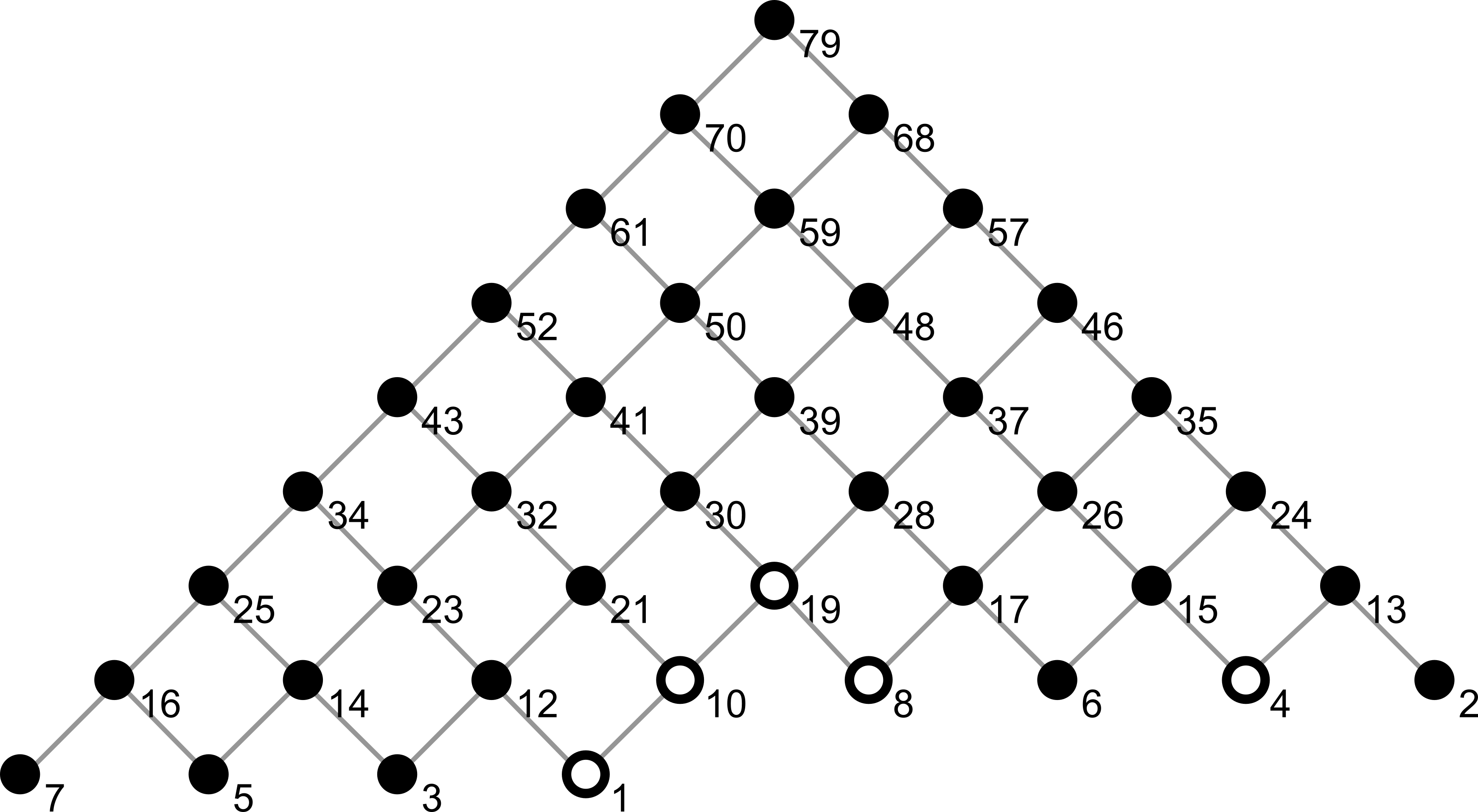}
\caption{The Hasse diagram of 
$P_{9, 11}$ with its order ideal $I = \{19, 10, 8, 4, 1\}$.}
\label{order-ideal-example}
\end{figure}

Using this identification, our main theorem can be rephrased as follows.

\begin{theorem}
\label{rephrase}
For every odd integer $s>0$, there is a bijection between 
the set of order ideals of $P_{s, s+2}$
containing no consecutive integers 
and the set of lattice paths of length $s$ from $(0, 0)$ to $(s, y)$, where $y$ varies over the positive integers.
\end{theorem}


The goal of the rest of this paper is to prove Theorem \ref{rephrase}. First we provide an overview of the bijection in section 2, and then we prove it in section 3.

\section{Overview of the bijection}

In this section, we give the overview of the bijection of Theorem \ref{rephrase}. The following figure (Figure \ref{bijection-overview}) illustrates this bijection for $s=13$.
\begin{figure}[h]
\includegraphics[width=12cm]{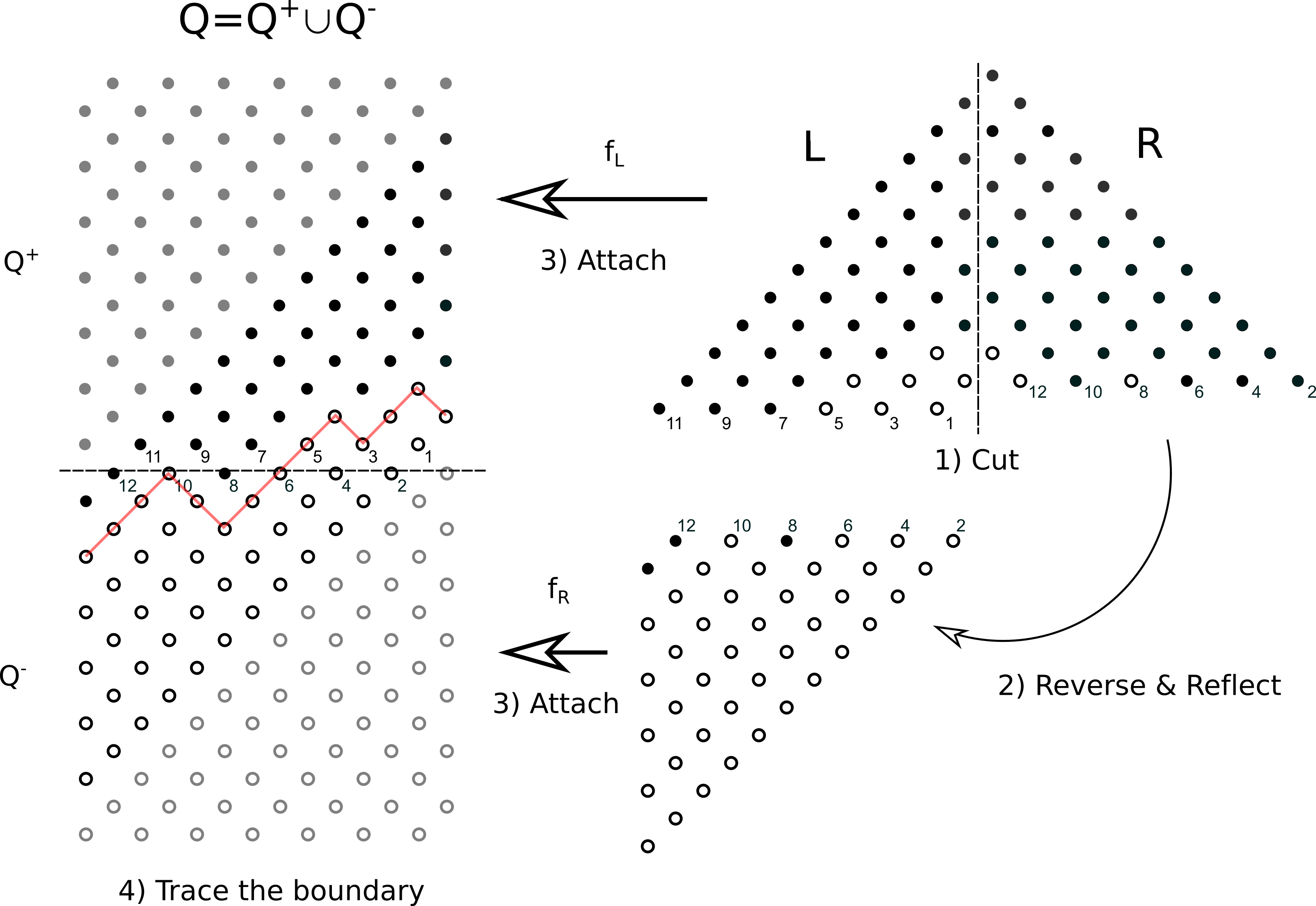}
\caption{An overview of the bijection of Theorem \ref{rephrase}}
\label{bijection-overview}
\end{figure}

The dots in the upper right corner of Figure \ref{bijection-overview} denote the poset $P_{s, s+2}$. 
The elements of our order ideal $I$ of $P_{s, s+2}$ containing no consecutive integers are colored white, and the others are colored black.

We first start by cutting the poset $P_{s, s+2}$ into two pieces in the middle (Step 1). Then we take the right part of the poset, switch the colors of dots (white to black and vice versa) and reflect everything horizontally (Step 2).
Then we assemble the left part and the reflected right part so that the integer $1$ is to the one step up and to the right of the integer $2$ (Step 3). It follows that any two adjacent integers appearing as the minimal elements of the poset are as close as possible to each other. For example, two integers $5$ and $6$ in Figure \ref{bijection-overview} are away by one `unit distance' after the assembling. 

Then $I$ corresponds to an ideal $J$ (collection of white dots) in the new poset $Q$, depicted on the left side of Figure \ref{bijection-overview}. The partial order of $Q$ is given in Definition \ref{Qposet}. The ideal $J$ satisfies a condition we call \emph{balanced}, and it turns out that the map is bijective (Theorem \ref{hard}).
Finally, we show that there is a bijection between
the set of balanced order ideals and the set
of lattice paths of length $s$
from $(0, 0)$ to $(s, y)$, where $y$ varies over the positive integers, by tracing the `\emph{boundary}' $P$ of the balanced order ideal $J$ (Step 4, Theorem \ref{bijection-dyck-path}).

\section{Proof of Theorem \ref{rephrase}}

The following description of $P_{s, t}$ when $s$ and $t$ are coprime was used by Anderson \cite{And02} to prove that the number
of $(s, t)$-core partitions is the generalized Catalan number $\frac{1}{s + t} \binom{s + t}{s}$. Let $F_{s,t} = st-s-t$ denote the Frobenius number, the largest positive integer that cannot be written as a linear combination of $s$ and $t$ with non-negative coefficients.

\begin{lemma}
\label{thm-pst-unique-repr}
Let $s$ and $t$ be positive coprime integers. Then 
\[
P_{s,t} = \{x \in \NN \mid x = F_{s, t} - as - bt \textnormal{ for some } a, b \in \NN\}.
\]
Any element $x \in P_{s, t}$ is \emph{uniquely} represented as $x = F_{s, t} - as - bt$ for some $a, b \in \NN$. Define the ordering of $P_{s,t}$ so that $ F_{s, t} - a_1s - b_1t \le F_{s, t} - a_0s - b_0t$ if and only if
$a_0 \leq a_1$ and $b_0 \leq b_1$.
\end{lemma}

\begin{definition}
\label{defn-new-poset}
Let 
\[
P'_{s, t} = \{(a, b) \in \NN \mid as + bt \leq F_{s, t} \}
\]
with the following poset structure: if $(a_0, b_0), (a_1, b_1) \in P'_{s, t}$, then $(a_1, b_1) \le (a_0, b_0)$ if and only if $a_0 \leq a_1$ and $b_0 \leq b_1$.
\end{definition}

By Lemma \ref{thm-pst-unique-repr} and Definition \ref{defn-new-poset},
it is easily seen that the map 
$$\psi \colon P_{s, t} \rightarrow P'_{s, t}$$ 
with $\psi(F_{s, t} - as - bt) = (a,b)$ is an isomorphism of posets. 

\begin{remark}
The posets $P_{s,t}$ and $P'_{s,t}$ are essentially the same. However, we mainly work with $P'_{s,t}$ since it helps to keep track of elements easily.
\end{remark}

For our main purpose, we fix $s = 2k+1$, and $t=2k+3$ for $k\in \NN$. The following lemma gives a complete description of the elements of $P'_{2k+1, 2k+3}$.

\begin{lemma}
\label{p2k12k3-desc}
We have
$$
P'_{2k+1, 2k+3} = \{(a, b) \in \NN \mid 
a + b \leq 2k - 1 \textnormal{ if } b \geq k \textnormal{ and }
a + b \leq 2k \textnormal{ if } b < k \}.
$$
\end{lemma}
\begin{proof}
By Definition \ref{defn-new-poset}, $(a, b)$ is in $P'_{2k+1, 2k+3}$ if and only if
$$(2k+1)a + (2k+3)b \leq F_{2k+1, 2k+3} = (2k+1)(2k+3) - (2k+1)-(2k+3),$$
or
\begin{equation*}
a+b \leq \frac{(2k+1)^2 - 2 - 2b}{2k+1}.
\end{equation*}
Now the lemma follows. 
\end{proof}

We divide $P'_{2k+1, 2k+3}$ into the `left part' $L$ and `right part' $R$ as follows.
\begin{definition}
Let 
$$
L = \{(a, b) \in P'_{2k+1,2k+3} \mid a > b \textnormal{ and } a + b \leq 2k\}
$$
and
$$
R = \{(a, b) \in P'_{2k+1,2k+3} \mid a \leq b \textnormal{ and } a + b \leq 2k - 1\}.
$$
\end{definition}
Then by Lemma \ref{p2k12k3-desc}, $P'_{2k+1, 2k+3}$ is a disjoint union of sets $L$ and $R$.
The sets $L$ and $R$ have naturally induced poset structures. The following picture illustrates $P_{9,11}$ and $P'_{9,11}$ ($k=4$). 

\begin{figure}[h]
\includegraphics[width=12cm]{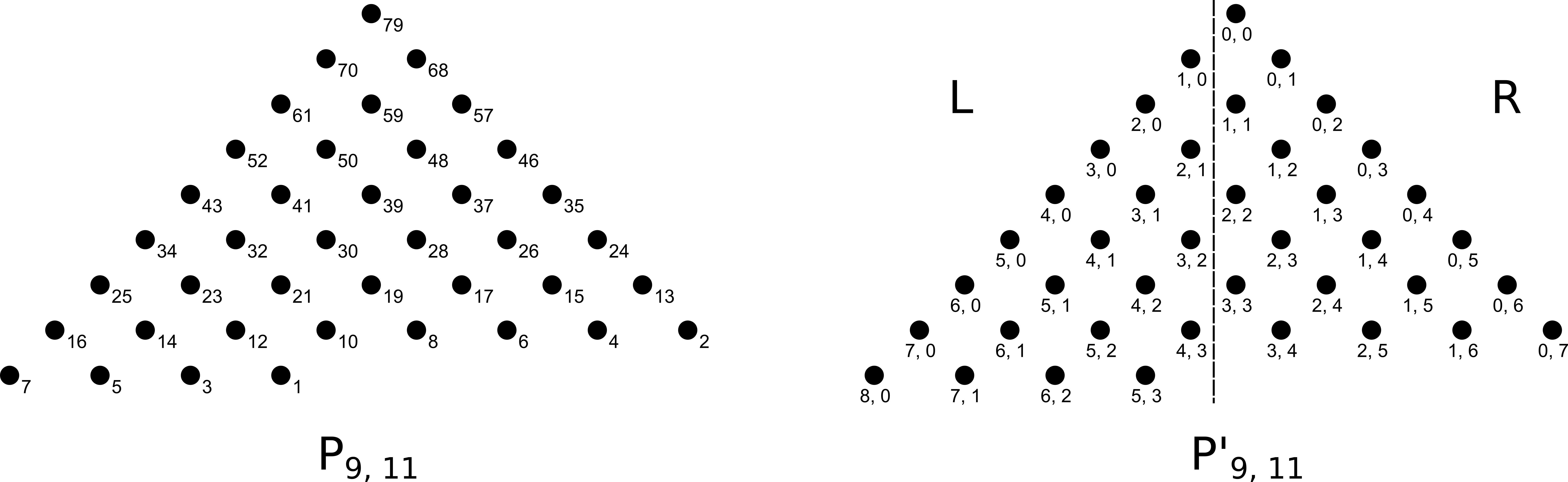}
\caption{$P_{9, 11}$ and $P'_{9, 11}$}
\label{poset-figure}
\end{figure}

\begin{lemma}
Let $l_i = (k + i, k - i) \in L$ for all $1 \leq i \leq k$,
and $r_j = (j - 1, 2k - j) \in R$ for all $1 \leq j \leq k$.
The $l_i$'s are all the minimal elements of $L$ and
$r_j$'s are all the minimal elements of $R$.
\end{lemma}

\begin{proof}
This lemma follows directly from Lemma \ref{p2k12k3-desc}.
\end{proof}

\begin{lemma}
\label{thm-small-numbers-in-p2k12k3}
Suppose that $x$ is an integer such that $1\le x \le 2k$. Then $x$ lies in $P_{2k+1, 2k+3}$. Moreover, if $x$ is odd, then $\psi(x) = l_{(x+1) / 2} \in L$.
If $x$ is even, then $\psi(x) = r_{x / 2} \in R$.
\end{lemma}
\begin{proof}
Recall that $F_{2k+1, 2k+3} = (2k+1)(2k+3)-(2k+1)-(2k+3) = 4k^2+4k-1$. The lemma follows from the following observation: For $1 \leq x \leq 2k$, we have
\[
x =
\begin{cases}
F_{2k+1, 2k+3} - (k+(x+1)/2)(2k+1) - (k-(x+1)/2)(2k+3) & \text{if $x$ is odd}, \\
F_{2k+1, 2k+3} - (x / 2 - 1)(2k+1) - (2k - x / 2)(2k+3) & \text{if $x$ is even}.
\end{cases}
\]
\end{proof}

In the following lemma, we rewrite the condition that `an order ideal in $P_{2k+1,2k+3}$ has no two adjacent numbers' in terms of the corresponding order ideal in $P'_{2k+1,2k+3}$. We use this condition in the proof of Theorem \ref{rephrase}.

\begin{lemma}
\label{ideal-characterization}
An order ideal $I$ 
of $P_{2k+1, 2k+3}$ contains no two adjacent numbers
if and only if the order ideal $I' = \psi(I)$ of $P'_{2k+1,2k+3}$ contains no two elements of the form $l_{x_0}$ and $r_{x_0}$
for some $1 \leq x_0 \leq k$, or of the form $r_{x_1}$ and $l_{x_1 + 1}$ for some $1 \leq x_1 < k$.
\end{lemma}

\begin{proof}
Note that $2k+1 \not\in P_{2k+1, 2k+3}$. For otherwise, there would exist $a,b \in \NN$ with $2k+1 = F_{2k+1, 2k+3} - a(2k+1)-b(2k+3)$ and therefore, $F_{2k+1,2k+3}$ would be written as a linear combination of $2k+1$ and $2k+3$ with non-negative coefficients. The fact that $I$ is an order ideal implies that if $x, x+1 \in I$, there exists $1 \leq x' \leq 2k + 1$ such that $x' \equiv x \pmod{2k+1}$ and $x',x'+1 \in I$. Since $2k+1$ is not in $P_{2k+1, 2k+3}$, we have $1 \leq x' \leq 2k - 1$. Now the lemma immediately follows from Lemma \ref{thm-small-numbers-in-p2k12k3}.
\end{proof}

To establish the bijection stated in Theorem \ref{rephrase},
we first construct the maps from the poset $P'_{2k+1, 2k+3}$
to the poset $Q$, which is described as follows.

\begin{definition}
\label{Qposet}
We define $Q$ to be the poset 
$$
Q = \{ (a, b) \in \ZZ^2 \mid 1 \leq a - b \leq 2k + 2 \},
$$
whose partial order is given by requiring $(a, b) \ge (a', b')$ if $(a', b') = (a + 1, b)$ or $(a', b') = (a, b + 1)$. Define 
\begin{align*}
Q^+ &= \{ (a, b) \in Q \mid a + b \leq 2k \}, \text{ and} \\
Q^-  &= \{ (a, b) \in Q \mid a + b \geq 2k + 1 \}.
\end{align*}
\end{definition}
It is evident that $Q$ is a disjoint union of $Q^+$ and $Q^-$. Definition \ref{Qposet} implies that we cannot have $x \in Q^+$ and $y \in Q^-$ with $x \le y$. 
A Hasse diagram of $Q$ for $k = 6$ is in
the left side of Figure \ref{bijection-overview}.
We introduce the notion of \emph{height}, and \emph{balanced} order ideal of $Q$. Let $J$ denote an order ideal of $Q$ for the rest of the paper. 
\begin{definition}
Let the \emph{height} $h_J(p)$ of $J$ at position $p$ ($0 \leq p \leq 2k + 1$) be
the maximum value of $2k + 1 - a - b$ where $(a, b) \in J$ and $a - b = 2k + 2 - p$.
This is well-defined for all proper, non-empty order ideals $J$ of $Q$. Let the \emph{left height} of $J$ be the height of $J$ at position $0$, i.e., $h_J(0)$. Let the \emph{right height} of $J$ be the height of $J$ at position $2k + 1$, i.e., $h_J(2k+1)$. 

We say a proper, non-empty order ideal $J$ of $Q$ is \emph{balanced} if the following holds.
\begin{enumerate}
\item
$h_J(0) < 0$,
\item
$h_J(2k+1) \ge 0$,
 \item
$|h_J(0)+h_J(2k+1)| = 1$.
\end{enumerate}
Note that $h_J(0)$ is odd and $h_J(2k+1)$ is even.
\end{definition}

\begin{remark}
\label{help-understanding}
Let $J_p$ be the set of elements of $J$ at position $p$ with the induced poset structure, i.e., $J_p:= \{(a,b) \in J | a-b = 2k+2-p\}$. In particular, $J$ is disjoint union of the sets $J_p$. Then there exists $(a_p, b_p) \in J_p$ such that $h_J(p)=2k+1-a_p-b_p$. One can check without difficulty that $(a_p, b_p)$ is the maximal element of $J_p$, so a maximal element of $J$. An order ideal in a poset is uniquely determined by its maximal elements. Therefore if heights $h_J(0), h_J(1), \cdots, h_J(2k+1)$ are given, then $J$ is uniquely determined. 
\end{remark}
This observation leads to the following proposition, whose details are left to the reader.

\begin{proposition}
\label{sequence-correspondence}
Any proper, non-empty order ideal $J$ of $Q$ determines the sequence of heights $h_J(0), h_J(1), \cdots, h_J(2k+1)$. We have that $h_i = h_J(i) \equiv i+1 \pmod{2}$ and $|h_{i+1} - h_{i}| = 1$. Conversely, for any sequence of integers $h_0, h_1, \cdots, h_{2k+1}$
such that $h_i \equiv i+1 \pmod{2}$ and $|h_{i+1} - h_{i}| = 1$, there is a unique proper non-empty order ideal $J$ of $Q$
such that $h_J(i) = h_i$.

\end{proposition}

Define the maps $f_L$ from $L$ to $Q^+$ 
and $f_R$ from $R$ to $Q^-$ as follows.
\begin{definition}
Let the map $f_L \colon L \ra Q^+$ be $f_L((a, b)) = (a, b)$ and
the map $f_R \colon R \ra Q^-$ be $f_R((a, b)) = (3k + 1 - b, k - 1 - a)$.
\end{definition}
The maps are easily checked to be well-defined. Note that if $(a,b) \le (c,d)$ in $R$, then $f_R((a,b)) \ge f_R((c,d))$ in $Q$. That is, $f_R$ reverses the order. We also see that $f_L$ preserves the order. 

As explained in Section 2, we complete the proof of Theorem \ref{rephrase} by showing the following two theorems.

\begin{theorem}
\label{hard}
The correspondence
\begin{equation}
\label{mapmap}
I \mapsto J = f(I) := f_L(I \cap L) \cup (Q^- \setminus f_R(I \cap R))
\end{equation}
is a one-to-one correspondence between the order ideals $I$ of $P'_{2k+1, 2k+3}$ satisfying the 
condition in Lemma \ref{ideal-characterization} and 
the balanced order ideals $J$ of $Q$.
\end{theorem}

\begin{proof}
There are three things that we need to check:
\begin{enumerate}
\item
$J = f(I)$ is an order ideal,
\item
$J$ is balanced,
\item
the map is one-to-one and onto.
\end{enumerate}

(1) First we show that $J=f(I)$ is an order ideal. Let 
\begin{align*}
J^+ &:= J \cap Q^+ = f_L(I \cap L), \\
J^- & := J \cap Q^- = Q^- \setminus f_R(I \cap R).
\end{align*}
We argue by contradiction: Suppose that $J$ is not an order ideal. Since $f_L$ preserves the poset structure, $J^+$ is an order ideal of $Q^+$. Likewise, $f_R$ reverses the order, so $J^-$ is an order ideal of $Q^-$. Hence, there must exist $x \in J^+$ and $y \in Q^-\setminus J^-$ such that $x \ge y$. By the structure of $Q$,
there is a sequence of elements $e_0 = x, e_1, \cdots, e_u = y$ in $Q$ such that
$e_j$ covers $e_{j+1}$ and there exists a unique $i$ such that $e_i \in J^+$ and $e_{i+1} \in Q^-\setminus J^- = f_R(I \cap R)$. Then $e_i$ is a minimum of $J^+$. Thus, $e_i = f_L(l_{x_0}) = (k+x_0, k-x_0)$ for some $1 \leq x_0 \leq k$ (and $l_{x_0} \in I$). Then it follows $e_{i+1} = (k+x_0+1,k-x_0)$ or $e_{i+1}= (k+x_0, k-x_0+1)$ since $e_i$ covers $e_{i+1}$. Together with the fact that $e_{i+1} \in f_R(I \cap R)$, this would imply that    $r_{x_0} \in I$ or $r_{x_0-1} \in I$ by the definition of $f_R$. However, this is impossible by Lemma \ref{ideal-characterization}. Therefore $J$ is an order ideal.

(2) Now we check that $J$ is balanced. Note that  $(k+1,k) \in Q^-$ is not in the image of $f_R$, so  $(k+1,k) \in J$. This implies that $h_J(2k+1) \ge 0$. We have $(2k+1,-1) \in Q^+\setminus f_L(I\cap L)$, so in particular $(2k+1,-1) \not\in J$. By Remark \ref{help-understanding}, we have $h_J(0) < 1$. Since $h_J(0)$ is odd, it must be negative.  

Now we show $|h_J(0) + h_J(2k+1)| = 1$.  Suppose that $(k,k-1) \not\in I$. Then clearly, $h_J(2k+1) =0$. Moreover $(k-1,k-1) \not\in I$ by the definition of order ideals. Hence $(2k+2,0)=f_R((k-1,k-1)) \in J$, so $h_J(0) = -1$, whence $|h_J(0) + h_J(2k+1)| = 1$. Suppose that $(k,k-1) \in I$. Let $\alpha$ be the smallest non-negative integer so that $(\alpha+1, \alpha) \in I$. Remark \ref{help-understanding} shows that $h_J(2k+1)=2k-2\alpha$. Let $\beta$ be the largest integer such that $(\beta,\beta) \in P'_{2k+1,2k+3} \setminus I$. Then $(3k+1-\beta, k-1-\beta) = f_R((\beta, \beta)) \in J$. By the maximality of $\beta$, we have $h_J(0) = 2k+1-(3k+1-\beta)-(k-1-\beta) = -2k + 2\beta + 1$. We also have $\beta = \alpha$ or $\beta = \alpha-1$ by the fact that $I$ is an order ideal. Then it follows immediately that $|h_J(0) + h_J(2k+1)| = 1$.

(3) Any balanced order ideal $J$ of $Q$ gives rise to an order ideal $I = f_L^{-1}(J \cap Q^+) \cup f_R^{-1}(Q^- \setminus J)$, which satisfies the condition in Lemma \ref{ideal-characterization}. Clearly, this defines the inverse of the map \eqref{mapmap}. Therefore, it follows that the map \eqref{mapmap} is a one-to-one correspondence.
\end{proof}

\begin{theorem}
\label{bijection-dyck-path}
The map
$$
J \mapsto P = \{(p, h_J(p) - h_J(0)) \mid 0 \leq p \leq 2k + 1\} 
$$
is a one-to-one correspondence between balanced order ideals $J$ of $Q$
and lattice paths $P$ from $(0, 0)$ to $(2k + 1, d)$, where $d$ varies over positive integers.
\end{theorem}
\begin{proof}
Proposition \ref{sequence-correspondence} implies that this map is well-defined. Take any path $P$ from $(0, 0)$ to $(2k + 1, d)$ for some positive $d$ (it follows that $d$ is odd). Let $p_i$ be the $y$-coordinate of the $i \textsuperscript{th}$ point in path $P$, that is, $(i, p_i) \in P$. Let $J_P$ be a subset of $Q$ so that $h_{J_P}(i) =  p_i - 2 \lfloor (d -1) / 4 \rfloor - 1$. One can see $J_P$ is indeed an order ideal by the criterion in Proposition \ref{sequence-correspondence}. Moreover, it is straightforward to check that $J_P$ is balanced and the correspondence takes $J_P$ to $P$. Hence the map is onto. 

Conversely, if a balanced order ideal $J$ maps to $P$, then $h_J(i)-h_J(0) = p_i$. Since $|h_J(0) + h_J(2k+1)| = 1$, it follows that $h_J(i) = h_{J_P}(i)$ for all $0 \le i \le 2k+1$. Proposition \ref{sequence-correspondence} implies that $J = J_P$, completing the proof.  
\end{proof}

\section*{Acknowledgement}
We are very grateful to Nathan Kaplan for many valuable comments. We also thank Dennis Eichhorn for discussions.

\bibliographystyle{abbrv}
\bibliography{ref}

\end{document}